\documentclass[reqno]{amsart}
\usepackage{amsmath,amsthm,amssymb}
\usepackage[T2A]{fontenc}
\usepackage[utf8]{inputenc}
\usepackage[english]{babel}
\usepackage[pdftex]{hyperref}
\usepackage{cases}

\renewcommand{\le}{\leqslant}
\renewcommand{\ge}{\geqslant}
\newcommand{\dd}{\displaystyle}

\title[On the maximum of a function]{On the maximum of a function connected with the Green function of a focal boundary value problem}

\author{E.I. Bravyi\\Perm National Research Polytechnic University, Perm, Russia}
\thanks{The work was performed as part of the State Task of the Ministry of Science and Higher Education of the Russian Federation (project FSNM-2020-0028). This work was supported by the Russian Foundation for Basic Research (project No. 18-01-00332).}
\date{August 2020}
\email{bravyi@perm.ru}

\theoremstyle{plain}
\newtheorem{theorem}{Theorem}
\newtheorem{lemma}[theorem]{Lemma}

\begin{document}
\begin{abstract} It's proved that a function connected with the Green function of the even order symmetric focal boundary value problem takes its maximal value on the diagonal.
\end{abstract}

\maketitle

\section{Introduction.}

We consider a focal boundary value problem for the simplest ordinary differential equation with even order  $n=2k$:
\begin{equation}\label{e--1}
  \left\{\begin{array}{l}
    (-1)^{n-k}x^{(n)}(t)=f(t),\quad t\in[0,1],\\
    x^{(i)}(0)=0,\quad i=0,\ldots,k-1,\quad x^{(j)}(1)=0,\quad j=k,\ldots,n-1.
  \end{array}
  \right.
\end{equation}

This problem
is uniquely solvable, that is for every integrable function $f$, it has a unique solution
\begin{equation*}
  x(t)=\int_{0}^{1}G(t,s)f(s)\,ds, \quad t\in[0,1],
\end{equation*}
with the Green function
\begin{equation*}
  G(t,s)=\frac{1}{(m!)^2}\int_{0}^{\min(s,t)} (t-\tau)^m(s-\tau)^m\,d\tau,\quad t,s\in[0,1],
\end{equation*}
where $m=k-1$ \cite{1978-Keener}
(so, the function $G(t,s)$ is symmetric: $G(t,s)=G(s,t)$ for all ${t,s\in[0,1]}$).

Put
\begin{equation*}
  M(t,s)\equiv\sqrt{G(t,s)G(1,1)}-\sqrt{G(t,1)G(s,1)},\quad t,s\in[0,1].
\end{equation*}

The aim of the work is to prove the following technic assertion which is useful for proving of solvability of focal boundary value problems for functional differential equations, for example, for proving Theorem 1.2 in  \cite[p.~31]{Bravyi-1}.

\begin{theorem} The function $M(t,s)$ takes its maximum on the diagonal of the square  $(t,s)\in[0,1]\times[0,1]$:
\begin{equation*}
  \max\limits_{\substack{0\le s \le 1\\ 0\le t \le 1}} M(t,s) = M(t_0,t_0)
\end{equation*}
for some $t_0\in[0,1]$.
\end{theorem}

\section{Proof}

\textbf{1.} It's clear, the function $M(t,s)$ is symmetric:
\begin{equation*}
  M(s,t)=\sqrt{G(s,t)G(1,1)}-\sqrt{G(s,1)G(1,t)}=M(t,s),\quad t,s\in[0,1].
\end{equation*}
Moreover,
\begin{equation*}
  M(t,0)=M(t,1)=0,\quad t\in[0,1],\quad M(0,s)=M(1,s)=0,\quad s\in[0,1].
\end{equation*}

For $m=0$, we have
\begin{equation*}
  M(t,s)= \left\{
                 \begin{array}{ll}
                   \sqrt{t}(1-\sqrt{s}), & \hbox{$0\le t\le s\le1$,} \\
                   \sqrt{s}(1-\sqrt{t}), & \hbox{$0\le s< t\le1$,}
                 \end{array}
               \right.
\end{equation*}
and
\begin{equation*}
  \max\limits_{\substack{0\le s \le 1\\ 0\le t \le 1}} M(t,s) =
  M\left(\frac14,\frac14\right).
\end{equation*}
Further we assume that integer $m\ge1$.

Now we show that $M(t,s)>0$ for all $t,s\in(0,1)$.
\newcommand{\GG}[5]{#1\left(
  \begin{array}{cc}
                   #2 & #3 \\
                   #4 & #5 \\
                 \end{array}
  \right)}
We have
\begin{equation*}
  M(t,s)=\frac{
               \GG{G}{t}{1}{s}{1}
              }
             {M_1(t,s)},\quad t,s\in(0,1),
\end{equation*}
where
\begin{equation*}
  \GG{G}{t}{1}{s}{1}\equiv
  \left|
                 \begin{array}{cc}
                   G(t,s) & G(t,1) \\
                   G(1,s) & G(1,1) \\
                 \end{array}
               \right|,
\end{equation*}
\begin{equation*}
  M_1(t,s)\equiv\sqrt{G(t,s)G(1,1)}+\sqrt{G(t,1)G(s,1)}>0,\quad t,s\in(0,1).
\end{equation*}

The Green function Грина $G(t,s)$ of the boundary value problem \eqref{e--1} is a oscillation kernel \cite[Theorem 8 (Kalafati-Gantmacher-Krein) or Theorem 9]{LS1976-2}. Therefore, in particular,
for all $\tau_1$, $\tau_2$, $s_1$, $s_2$ such that $0<\tau_1<\tau_2\le 1$, $0<s_1<s_2\le 1$, the following inequality is valid:
\begin{equation*}
  \GG{G}{\tau_1}{\tau_2}{s_1}{s_2}=\left|
    \begin{array}{cc}
      G(\tau_1,s_1) & G(\tau_1,s_2) \\
      G(\tau_2,s_1) & G(\tau_2,s_2)\\
    \end{array}
  \right|>0.
\end{equation*}
Therefore, $\GG{G}{t}{1}{s}{1}>0$ for all $t,s\in(0,1)$. It follows that  $M(t,s)>0$ for all $t,s\in(0,1)$.

Note, that the function $M_0(t,s)\equiv\GG{G}{t}{1}{s}{1}$ takes its maximal value just on the diagonal. Indeed, $\frac{M_0(t,s)}{G(1,1)}$ is the Green function of the problem
\begin{equation}\label{L-5}
  \left\{
  \begin{array}{l}
  (-1)^{n-k}x^{(n)}(t)=f(t),\quad t\in[0,1],\\
  x^{(i)}(0)=0,\quad i=0,\ldots,k-1,\\
  x(1)=0,\\
  \text{if $n\ge2$, then }x^{(j)}(1)=0,\quad j=k,\ldots,n-2.\\
  \end{array}
  \right.
\end{equation}
The Green function of \eqref{L-5} is also the oscillation kernel (by
the Kalafati-Gantmacher-Krein theorem \cite[теорема 8]{LS1976-2} or by Theorem 9 from the same article. Therefore, in particular, the inequality
\begin{equation*}
  \GG{M_0}{s}{t}{s}{t}>0,\quad 0<s<t<1,
\end{equation*}
holds. This inequality means that
\begin{equation*}
  M_0(t,s)<\sqrt{M_0(t,t)M_0(s,s)}\le \max\{M_0(t,t),M_0(s,s)\},\quad 0<s<t<1.
\end{equation*}
Thus, the function $M_0(t,s)$ takes its maximum value on the diagonal of the square (for some $t=s\in(0,1)$).

However, a such short proof for the function $M(t,s)$ is unknown for us. For this function, we divide the proof into some steps. First, we prove inequality \eqref{L-10} (item \textbf{2}). It follows that the extremal points of $M(t,s)$ can be placed only on the diagonal. Inequality \eqref{L-10} with respect the variables $t$ and $s$ is reduced to inequality \eqref{L-11} with respect the variables $s$, $t$, $s/t$ (item \textbf{3}), which is equivalent to the increasing of the function $F(y)$ from equality \eqref{F(y)} for all $y\in(0,1)$ (items \textbf{4}, \textbf{5}).
The increasing of $F(y)$ is equivalent to the increasing of the function $X_m(x)$ from equality \eqref{X_m(x)} for $x>0$ (items \textbf{6}, \textbf{7}). The function $X_m$ is fraction \eqref{X-frac} whose  numerator and denominator are polynomials with positive coefficients. The increasing of $X_m$ is equivalent to the increasing of the consequence of the relations of the corresponding coefficients of these polynomials.
The coefficients have the representation \eqref{A-11}, \eqref{A-13} via finite sums (item \textbf{8}). With this representation, the increasing of the consequence can be checked (Lemma \ref{Lemma-3}).

\textbf{2.} Suppose the maximum of the positive in $[0,1]\times[0,1]$ function $M(t,s)$ is taken at the point $(t_0,s_0)\in(0,1)\times(0,1)$:
\begin{equation*}
  \max_{t,s\in[0,1]} M(t,s)=M(t_0,s_0).
\end{equation*}

For  $t,s\in(0,1)$, we have
\begin{equation*}
  \frac{\partial M(t,s)}{\partial s}=\frac12 \frac{G'_s(t,s)}{G(t,s)}\sqrt{G(t,s)G(1,1)}-\frac12 \frac{G'_s(1,s)}{G(1,s)}\sqrt{G(1,s)G(t,1)},
\end{equation*}
\begin{equation*}
  \frac{\partial M(t,s)}{\partial t}=\frac12 \frac{G'_t(t,s)}{G(t,s)}\sqrt{G(t,s)G(1,1)}-\frac12 \frac{G'_t(t,1)}{G(t,1)}\sqrt{G(1,s)G(t,1)},
\end{equation*}
where
\begin{equation*}
  G'_t(t,s)\equiv\frac{\partial G(t,s)}{\partial t},\quad  G'_s(t,s)\equiv\frac{\partial G(t,s)}{\partial s}.
\end{equation*}
Further we also use the notation
\begin{equation*}
  G''_{ss}(t,s)\equiv\frac{\partial^2G(t,s)}{\partial s^2}.
\end{equation*}

By the equalities
\begin{equation*}
  \frac{\partial M(t_0,s_0)}{\partial s}=0,\quad \frac{\partial M(t_0,s_0)}{\partial t}=0,
\end{equation*}
it follows that
\begin{equation*}
  \dd\frac{\dd\frac{G'_t(t_0,s_0)}{G(t_0,s_0)}}{\dd\frac{G'_t(t_0,1)}{G(t_0,1)}}=\sqrt{\frac{G(1,s_0)G(t_0,1)}{G(t_0,s_0)G(1,1)}}=
  \dd\frac{\dd\frac{G'_s(t_0,s_0)}{G(t_0,s_0)}}{\dd\frac{G'_s(1,s_0)}{G(1,s_0)}},
\end{equation*}
therefore,
\begin{equation*}
\dd\frac{\dd\frac{G(1,s_0)}{G'_s(1,s_0)} } {\dd\frac{G(t_0,1)}{G'_t(t_0,1)}}=
\dd\frac{G'_t(t_0,s_0)}{G'_s(t_0,s_0)}.
\end{equation*}

We will show that
\begin{equation}\label{L-10}
\dd\frac{\dd\frac{G(1,s)}{G'_s(1,s)} } {\dd\frac{G(t,1)}{G'_t(t,1)}}>
\dd\frac{G'_t(t,s)}{G'_s(t,s)}\quad\text{for all}\quad  0<s<t<1.
\end{equation}

It will follows that the maximum of $M$ cannot be taken in the triangle $0\le s<t\le1$ and (by the symmetry) in the triangle $0\le t<s\le1$. Thus, the maximum of $M(t,s)$ is taken at some  $t=s\in(0,1)$.

\textbf{3.} To prove  \eqref{L-10} we first obtain
representations (as finite sums) for all functions from this inequality. For all  $0<s<t<1$, we have
\begin{equation*}
  (m!)^2G(1,s)=\int_{0}^{s}(s-\tau)^m(1-\tau)^m\,d\tau,\quad
\end{equation*}
\begin{equation*}
  (m!)^2G(t,1)=\int_{0}^{t}(t-\tau)^m(1-\tau)^m\,d\tau,\quad
\end{equation*}
\begin{equation*}
  (m!)^2G'_s(1,s)=m\int_{0}^{s}(s-\tau)^{m-1}(1-\tau)^m\,d\tau,\quad
\end{equation*}
\begin{equation*}
  (m!)^2G'_t(t,1)=m\int_{0}^{t}(t-\tau)^{m-1}(1-\tau)^m\,d\tau,\quad
\end{equation*}
\begin{equation*}
  (m!)^2G'_t(t,s)=m\int_{0}^{s}(s-\tau)^m(t-\tau)^{m-1}\,d\tau.\quad
\end{equation*}
For integer non-negative $i$, $j$, by induction with respect to $j$, we see that
\begin{equation*}
  \int_{0}^{1}\theta^i(1-\theta)^j\,d\theta=\frac{i!j!}{(i+j+1)!}.
\end{equation*}

From this, we get the representation, for example, for the function $G'_s$:
\begin{equation*}
\begin{split}
(m!)^2G'_s(t,s)=m\int_{0}^{s}(t-\tau)^m(s-\tau)^{m-1}\,d\tau
\\=
m\,s^{2m}\int_{0}^{s}
\left(\frac{t}{s}-\frac{\tau}{s}\right)^m\left(1-\frac{\tau}{s}\right)^{m-1}\,\frac{d\tau}{s}
\\=
m\,s^{2m}\int_{0}^{1}
\left(\frac{t}{s}-\theta\right)^m\left(1-\theta\right)^{m-1}\,d\theta
\\=
m\,s^{2m}\sum_{i=0}^m \left(\frac{t}{s}\right)^{m-i}(-1)^i \binom{m}{i}\frac{i!(m-1)!}{(i+m)!}
\\=
s^{2m}\sum_{i=0}^m \left(\frac{t}{s}\right)^{m-i}(-1)^i \binom{2m}{m-i}\frac{(m!)^2}{(2m)!}
\\=
s^{m}t^{m}\sum_{i=0}^m \left(\frac{s}{t}\right)^{i}(-1)^i \binom{2m}{m-i}\frac{(m!)^2}{(2m)!}.
\end{split}
\end{equation*}
Similarly, the following equalities can be obtained (for all $0<s<t<1$):
\begin{equation*}
\begin{split}
G(t,s)
=s^{m+1}t^{m}\sum_{i=0}^{m} \left(\frac{s}{t}\right)^{i}(-1)^i
\binom{2m+1}{m-i}\frac{1}{(2m+1)!},
\end{split}
\end{equation*}
\begin{equation*}
\begin{split}
G'_t(t,s)=
\frac{m\,t^{m-1}s^{m+1}}{(m!)^2}\int_{0}^{1}
\left(1-\frac{s}{t}\theta\right)^{m-1}\left(1-\theta\right)^{m}\,d\theta
\\=s^{m+1}t^{m-1}\sum_{i=0}^{m-1} \left(\frac{s}{t}\right)^{i}(-1)^i \binom{2m}{m-1-i}\frac{1}{(2m)!},
\end{split}
\end{equation*}
\begin{equation*}
\begin{split}
G'_s(t,s)=
\frac{m\,t^{m}s^{m}}{(m!)^2}\int_{0}^{1}
\left(1-\frac{s}{t}\theta\right)^{m}\left(1-\theta\right)^{m-1}\,d\theta
\\
=s^{m}t^{m}\sum_{i=0}^{m} \left(\frac{s}{t}\right)^{i}(-1)^i \binom{2m}{m-i}\frac{1}{(2m)!},
\end{split}
\end{equation*}
\begin{equation}\label{A-1}
\begin{split}
G(1,s)=
\frac{s^{m+1}}{(m!)^2}\int_{0}^{1}
\left(1-s\theta\right)^{m}\left(1-\theta\right)^{m}\,d\theta
\\
=s^{m+1}\sum_{i=0}^{m} s^{i}(-1)^i \binom{2m+1}{m-i}\frac{1}{(2m+1)!},
\end{split}
\end{equation}
\begin{equation}\label{A-2}
\begin{split}
G'_s(1,s)=
\frac{ms^{m}}{(m!)^2}\int_{0}^{1}
\left(1-s\theta\right)^{m}\left(1-\theta\right)^{m-1}\,d\theta
\\
=s^{m}\sum_{i=0}^{m} s^{i}(-1)^i \binom{2m}{m-i}\frac{1}{(2m)!},
\end{split}
\end{equation}
\begin{equation*}
\begin{split}
G(t,1)=
t^{m+1}\sum_{i=0}^{m} t^{i}(-1)^i \binom{2m+1}{m-i}\frac{1}{(2m+1)!},
\end{split}
\end{equation*}
\begin{equation*}
\begin{split}
G'_t(t,1)
=t^{m}\sum_{i=0}^{m} t^{i}(-1)^i \binom{2m}{m-i}\frac{1}{(2m)!}.
\end{split}
\end{equation*}

Consider the right hand side of inequality \eqref{L-10}. For $0<s<t<1$, we have
\begin{equation*}
\dd\frac{G'_t(t,s)}{G'_s(t,s)}=\frac{t^{m-1}s^{m+1}}{t^{m}s^{m}}\frac{\sum_{i=0}^{m-1} \left(\frac{s}{t}\right)^{i}(-1)^i \binom{2m}{m-1-i}}{\sum_{i=0}^{m} \left(\frac{s}{t}\right)^{i}(-1)^i \binom{2m}{m-i}}=\frac{s}{t}\widetilde{Q} \left(\frac{s}{t}\right),
\end{equation*}
where
\begin{equation*}
\widetilde{Q}(s)=\frac{\sum_{i=0}^{m-1} s^{i}(-1)^i \binom{2m}{m-1-i}}{\sum_{i=0}^{m} s^{i}(-1)^i \binom{2m}{m-i}}.
\end{equation*}
Obtain a representation for the left-hand side of \eqref{L-10}. Denote
\begin{equation}\label{A-3}
\begin{split}
  Q(s)\equiv \dd\frac{G(1,s)}{G'_s(1,s)} =\frac{s^{m+1}\sum_{i=0}^{m} s^{i}(-1)^i \binom{2m+1}{m-i}\frac{1}{(2m+1)!}}{s^{m}\sum_{i=0}^{m} s^{i}(-1)^i \binom{2m}{m-i}\frac{1}{(2m)!}}\equiv\frac{s}{2m+1} \widehat{Q}(s),
\end{split}
\end{equation}
where
\begin{equation*}
\begin{split}
  \widehat{Q}(s)\equiv \dd\frac{\sum_{i=0}^{m} s^{i}(-1)^i \binom{2m+1}{m-i}}{\sum_{i=0}^{m} s^{i}(-1)^i \binom{2m}{m-i}}.
\end{split}
\end{equation*}
Obviously, that  $Q(s)>0$, $\widehat{Q}(s)>0$ for $s\in(0,1)$.
Moreover, $\widetilde{Q}(s)+1=\widehat{Q}(s)$. Indeed,
\begin{equation*}
  \widetilde{Q}(s)+1-\widehat{Q}(s)=\dd\frac{\sum_{i=0}^{m} s^{i}(-1)^i \left(\binom{2m}{m-1-i}+\binom{2m}{m-i}-\binom{2m+1}{m-i}\right)}{\sum_{i=0}^{m} s^{i}(-1)^i \binom{2m}{m-i}}=0.
\end{equation*}
So, inequality \eqref{L-10} is equivalent to the inequality
\begin{equation*}
  \frac{Q(s)}{Q(t)}=\frac{s}{t}\frac{\widehat{Q}(s)}{\widehat{Q}(t)}>\frac{s}{t}\widetilde{Q}\left(\frac{s}{t}\right),
\end{equation*}
that is equivalent to the inequality
\begin{equation}\label{L-11}
  \frac{\widehat{Q}(s)}{\widehat{Q}(t)}>\widehat{Q}\left(\frac{s}{t}\right)-1,\quad 0<s<t<1.
\end{equation}

\textbf{4.} Consider inequality  \eqref{L-11} for $t=1$. We have
\begin{equation}\label{Q(1)=2}
\begin{split}
\widehat{Q}(1)=
\dd\frac{\sum_{i=0}^{m} (-1)^i \binom{2m+1}{i}}
        {\sum_{i=0}^{m} (-1)^i \binom{2m}{i}  }
        =
   \frac{\binom{2m}{m}}{\binom{2m-1}{m}}=2,
\end{split}
\end{equation}
To prove \eqref{L-11} for $t=1$ it remains to prove the inequality
\begin{equation*}
  \widehat{Q}(s)<2,\quad s\in(0,1).
\end{equation*}

Here and further we use the following assertion.

\begin{lemma}\label{Lemma-1}
Let $A_{j}>0$, $B_{j}>0$ for $j=0,\ldots,J$.

If the sequence $\frac{A_j}{B_j}$, $j=0,\ldots,J$, doesn't decrease and  are not constant, the function
\begin{equation*}
  X(x)\equiv\frac{\sum_{j=0}^{J}A_{j}x^j}{\sum_{j=0}^{J}B_{j}x^j},\quad x>0,
\end{equation*}
increases.
\end{lemma}
\begin{proof}
It's enough to compute the derivative:
\begin{equation*}
\begin{split}
  \frac{dX}{dx}=\frac{{\sum_{j=0}^{J}A_{j}jx^j}{\sum_{i=0}^{J}B_{i}x^i}-{\sum_{i=0}^{J}A_{i}x^i}{\sum_{j=0}^{J}B_{j}jx^j}}{{x\left(\sum_{j=0}^{J}B_{j}x^j\right)^2}}
\\=\frac{{\sum_{0\le i<j\le J} B_jB_i \left(\frac{A_j}{B_j}-\frac{A_i}{B_i}\right)(j-i)x^{i+j}}}{{x\left(\sum_{j=0}^{J}B_{j}x^j\right)^2}}>0 \ \text{ for } x>0.
\end{split}
\end{equation*}
\end{proof}

From the equality \eqref{Q(1)=2} and from the next lemma, it follows that inequality \eqref{L-11} is fulfilled for $t=1$, $s\in(0,1)$.

First, it will be useful to compute the following integral.
For integer non-positive $M$ and $K$ we have
\begin{equation}\label{L-12}
\begin{split}
  \int_0^1(1-\theta)^K(1-\theta s)^M\,d\theta=\int_0^1(1-\theta)^K(s(1-\theta)+1-s)^M\,d\theta\\=
\sum_{i=0}^{M}s^i\int_0^1(1-\theta)^{i+K}\,d\theta\binom{M}{i}(1-s)^{M-i}=
\sum_{i=0}^{M}\frac{s^i(1-s)^{M-i}}{i+K+1}\binom{M}{i}
\\=
\sum_{i=0}^{M}\left(\frac{s}{1-s}\right)^i \frac{\binom{M}{i}(1-s)^M}{i+K+1}.
\end{split}
\end{equation}

\begin{lemma}\label{Lemma-2}
The function $\widehat{Q}(s)$ increases for $s\in(0,1)$.
\end{lemma}

\begin{proof}
Using \eqref{A-1} and \eqref{A-2}, for the function $\widehat{Q}$ defined in \eqref{A-2}, we obtain
\begin{equation*}
\begin{split}
  \widehat{Q}(s)=\frac{(2m+1)G(1,s)}{s\,G'_s(1,s)}=
\frac{2m+1}{m}
\frac{\int_{0}^{1}(1-s\theta)^{m}(1-\theta)^{m}\,d\theta}
     {\int_{0}^{1}(1-s\theta)^{m}(1-\theta)^{m-1}\,d\theta}.
\end{split}
\end{equation*}
Further, by \eqref{L-12} for  $M=m$ и $K=m$, $K=m-1$, we have
\begin{equation}\label{L-16}
\begin{split}
 \widehat{Q}(s)=\frac{2m+1}{m}
\frac{\sum_{i=0}^{m} \left(\frac{s}{1-s}\right)^{i}\binom{m}{i}\frac{1}{i+m+1}}
     {\sum_{i=0}^{m} \left(\frac{s}{1-s}\right)^{i}\binom{m}{i}\frac{1}{i+m}}.
\end{split}
\end{equation}
Now it's enough to apply Lemma \ref{Lemma-1}: the corresponding sequence
\begin{equation*}
  \frac{A_i}{B_i}=\frac{i+m}{i+m+1},\quad i=0,\ldots,m,
\end{equation*}
decreases. Therefore, the right-hand side of  \eqref{L-16} increases in the variable $\frac{s}{1-s}$ and the function $\widehat{Q}$ increases for $s\in(0,1)$.
\end{proof}

So, we proved that inequality \eqref{L-11} is strict for $t=1$, $s\in(0,1)$.

\textbf{5.} Now we will prove that for all $0<s<t<1$ inequality \eqref{L-11} holds if the function
\begin{equation}\label{F(y)}
  F(y)\equiv y\frac{\frac{dQ(y)}{dy}}{Q(y)}
\end{equation}
increases for $y\in(0,1)$, where the function $Q$ is defined by \eqref{A-3}.
Multiplication by a constant function or by a power function preserves increasing of $F$, therefore the function $F$ increases if the function
\begin{equation*}
  \widehat F(y)\equiv y\frac{\frac{d \widehat Q(y)}{dy}}{\widehat Q(y)}
\end{equation*}
increases.

Denote $\widehat{Q}'(y)\equiv\frac{d\widehat{Q}(y)}{dy}$.

Put $k=s/t\in(0,1)$. If the function $\widehat F(y)$ increases, then
\begin{equation*}
\begin{split}
  \frac{d\phantom{r}}{dt}\frac{\widehat{Q}(kt)}{ \widehat{Q}(t)}=
\frac{k\widehat{Q}'(kt)\widehat{Q}(t)-\widehat{Q}(kt)\widehat{Q}'(t)}{(\widehat{Q}(t))^2}=
  \left(\frac{kt\widehat{Q}'(kt)}{\widehat{Q}(kt)}-
       \frac{t\widehat{Q}'(t)}{\widehat{Q}(t)} \right)
 \frac{\widehat{Q}(kt)} {t\widehat{Q}(t)}
\\=
\left(\widehat{F}(kt)-\widehat{F}(t) \right)
 \frac{\widehat{Q}(kt)} {t\widehat{Q}(t)}<0,\quad t\in(s,1),
\end{split}
\end{equation*}
therefore, inequality \eqref{L-11} is valid for $s/t=k$, since it holds for $t=1$. The number $k\in(0,1)$ is arbitrary, therefore, inequality \eqref{L-11} holds for all  $0<s<t<1$.

\textbf{6.} Consider the case $m=1$. Then
\begin{equation*}
F(y)=\frac{y^2-4y+6}{(3-y)(2-y)},\quad F'(y)=\frac{6-y^2}{(3-y)^2(2-y)^2}>0,\quad y\in(0,1).
\end{equation*}
Thus, for $m=1$, the function $F$ decreases and inequalities \eqref{L-11},  \eqref{L-10}  are fulfilled.

\textbf{7.} Let further $m\ge2$.

We have
\begin{equation*}
  Q(s)=\frac{G(1,s)}{G'_s(1,s)},\quad Q'(s)=1-\frac{G(1,s)G''_{ss}(1,s)}{(G'_s(1,s))^2},
\end{equation*}
\begin{equation}\label{A-F}
  F(s)=\frac{sQ'(s)}{Q(s)}=s\left(\frac{G'_s(1,s)}{G(1,s)}-
\frac{G''_{ss}(1,s)}{G'_s(1,s)}\right).
\end{equation}

It remains to prove that the function $F(y)$ increases for $y\in(0,1)$. For this, it's convinient to get new representations for the functions $G(1,s)$, $G'_s(1,s)$, $G''_{ss}(1,s)$. The first two functions are defined by \eqref{A-1} and \eqref{A-2}. The function $G''_{ss}(1,s)$ is defined by the equality
\begin{equation*}
\begin{split}
G''_{ss}(1,s)=  m(m-1)\int_{0}^{s}(s-\tau)^{m-2}(1-\tau)^m\,d\tau\frac{1}{(m!)^2}\\=
\frac{m(m-1)s^{m-1}}{(m!)^2}\int_{0}^{1}
\left(1-s\theta\right)^{m}\left(1-\theta\right)^{m-2}\,d\theta.
\end{split}
\end{equation*}

We use \eqref{L-12} for computing $G(1,s)$ for  $K=m$, $M=m$, for computing $G'_s(1,s)$ when $K=m-1$, $M=m$ and  $G''_{ss}(1,s)$ for  $K=m-2$, $M=m$:
\begin{equation*}
\begin{split}
G(1,s)=\frac{s^{m+1}(1-s)^m}{(m!)^2}\sum_{i=0}^{m}\left(\frac{s}{1-s}\right)^i \frac{\binom{m}{i}}{i+m+1},\\
G'_s(1,s)=\frac{ms^{m}(1-s)^m}{(m!)^2}\sum_{i=0}^{m}\left(\frac{s}{1-s}\right)^i \frac{\binom{m}{i}}{i+m},\\
G''_{ss}(1,s)=\frac{m(m-1)s^{m-1}(1-s)^m}{(m!)^2}\sum_{i=0}^{m}\left(\frac{s}{1-s}\right)^i \frac{\binom{m}{i}}{i+m-1}.\\
\end{split}
\end{equation*}

Substituting the obtained expressions in \eqref{A-F}, we get the representation for the function $F(s)$, $s\in(0,1)$ which increasing should be proved:
\begin{equation*}
\begin{split}
  F(s)=m\frac{\sum_{i=0}^{m}\left(\frac{s}{1-s}\right)^i \frac{\binom{m}{i}}{i+m}}{\sum_{i=0}^{m}\left(\frac{s}{1-s}\right)^i \frac{\binom{m}{i}}{i+m+1}}-(m-1)\frac{\sum_{i=0}^{m}\left(\frac{s}{1-s}\right)^i \frac{\binom{m}{i}}{i+m-1}}{\sum_{i=0}^{m}\left(\frac{s}{1-s}\right)^i \frac{\binom{m}{i}}{i+m}}.
\end{split}
\end{equation*}

Introducing a new variable $x=\frac{s}{1-s}$, we obtain that we need to prove the increasing of the function
\begin{equation}\label{X_m(x)}
  X_m(x)=m\frac{\sum_{i=0}^{m}\frac{\binom{m}{i}}{i+m}x^i }{\sum_{i=0}^{m} \frac{\binom{m}{i}}{i+m+1}x^i}-
(m-1)
\frac{\sum_{i=0}^{m}\frac{\binom{m}{i}}{i+m-1}x^i}{\sum_{i=0}^{m}\frac{\binom{m}{i}}{i+m}x^i},\quad x>0.
\end{equation}

We have
\begin{equation*}
\begin{split}
\sum_{i=0}^{m} \frac{\binom{m}{i}}{i+m+1}x^i \sum_{i=0}^{m} \frac{\binom{m}{i}}{i+m}x^i
\\=
\sum_{k=0}^{2m} x^k \sum_{j=\max(0,k-m)}^{\min(k,m)} \binom{m}{j}\binom{m}{k-j}\left(\frac{1}{j+m+1}\frac{1}{k-j+m}\right)
\\=
\sum_{k=0}^{2m} x^k \sum_{j=0}^{m} \binom{m}{j}\binom{m}{k-j}\frac{1}{2m+k+1}\left(\frac{2(j+m+1/2)}
{(j+m+1)(j+m)}\right),
\end{split}
\end{equation*}
\begin{equation*}
\begin{split}
m\sum_{i=0}^{m} \frac{\binom{m}{i}}{i+m}x^i \sum_{i=0}^{m} \frac{\binom{m}{i}}{i+m}x^i-(m-1)\sum_{i=0}^{m} \frac{\binom{m}{i}}{i+m+1}x^i \sum_{i=0}^{m} \frac{\binom{m}{i}}{i+m-1}x^i
\\=
\sum_{k=0}^{2m} x^k \sum_{j=\max(0,k-m)}^{\min(k,m)} \binom{m}{j}\binom{m}{k-j}
\\\times
\left(\frac{m}{j+m}\frac{1}{k-j+m}-\frac{m-1}{j+m+1}
\frac{1}{k-j+m+1}\right)
\\=
\sum_{k=0}^{2m} x^k \sum_{j=0}^{m} \binom{m}{j}\binom{m}{k-j}\frac{1}{2m+k}\left(\frac{2(j+m+j/(j+m-1))}
{(j+m+1)(j+m)}\right).
\end{split}
\end{equation*}
Here, as usual, it's supposed that $\binom{m}{j}=0$ if $m\ge0$, $j<0$ or $j>m$.

\textbf{8.} Now the function $X_m(x)$ can be written in the following way:
\begin{equation}\label{X-frac}
  X_m(x)\equiv \frac{\sum_{k=0}^{2m} A_{k,m}x^k}{\sum_{k=0}^{2m}B_{k,m}x^k},\quad x>0,
\end{equation}
where
\begin{equation}\label{A-11}
A_{k,m}=\sum_{j=0}^{m} \binom{m}{j}\binom{m}{k-j}\frac{1}{2m+k}f_j\ge0,
\end{equation}
\begin{equation*}
f_j=\frac{2(j+m+j/(j+m-1))}{(j+m+1)(j+m)},
\end{equation*}
\begin{equation}\label{A-13}
B_{k,m}=\sum_{j=0}^{m} \binom{m}{j}\binom{m}{k-j}\frac{1}{2m+k+1}g_j\ge0,
\end{equation}
\begin{equation*}
g_j=\frac{2(j+m+1/2)}{(j+m+1)(j+m)}.
\end{equation*}

\begin{lemma}\label{Lemma-3} Let $m\ge2$.
The sequence $\frac{A_{k,m}}{B_{k,m}}$, $k=0,\ldots,2m$, increases.
\end{lemma}

\begin{proof}
It follows from equalities \eqref{A-11}, \eqref{A-13} that
\begin{equation*}
A_{k,m}=\sum_{j=0}^{m} \frac{R_{m,k,j}}{2m+k}f_j,\quad B_{k,m}=\sum_{j=0}^{m} \frac{R_{m,k,j}}{2m+k+1}g_j,
\end{equation*}
where
\begin{equation*}
  R_{m,k,j}\equiv\frac{m^{\underline{j}}\,k^{\underline{j}}\,m^{\underline{k-j}}}{k!j!},
\quad
  i^{\underline{j}}\equiv\left\{
           \begin{array}{ll}
             0, & \hbox{integer $j<0$,} \\
             1, & \hbox{$j=0$,} \\
             i(i-1)\ldots(i-j+1), & \hbox{integer $j\ge1$.}
           \end{array}
         \right.
\end{equation*}
Now we will compute the difference between neighbour terms in the sequence
$\frac{A_{k,m}}{B_{k,m}}$, $k=0,\ldots,2m$:
\begin{equation*}
  \frac{A_{k+1,m}}{B_{k+1,m}}-\frac{A_{k,m}}{B_{k,m}}=\frac{A_{k+1,m}-A_{k,m}B_{k+1,m}}{B_{k,m}B_{k+1,m}},
\end{equation*}

\begin{equation*}
\begin{split}
{A_{k+1,m}B_{k,m}-A_{k,m}B_{k+1,m}}=
\sum_{i=0}^m\frac{R_{m,k+1,i}}{2m+k+1}f_i
\sum_{j=0}^m\frac{R_{m,k,j}}{2m+k+1}g_j\\-
\sum_{j=0}^m\frac{R_{m,k,j}}{2m+k}f_j
\sum_{i=0}^m\frac{R_{m,k+1,i}}{2m+k+2}g_i
=\sum_{i=0}^m\sum_{j=0}^m R_{m,k+1,i}R_{m,k,j}\frac{f_ig_j}{(2m+k+1)^2}\\-
\sum_{i=0}^m\sum_{j=0}^m R_{m,k+1,i}R_{m,k,j}\frac{f_jg_i}{(2m+k)(2m+k+2)}
\\=
\sum_{i=1}^m\sum_{j=0}^{i-1}
\left(\frac{R_{m,k+1,i}R_{m,k,j}}{(2m+k+1)^2}-\frac{R_{m,k+1,j}R_{m,k,i}}{(2m+k)(2m+k+2)}\right)(f_ig_j-f_jg_i).
\end{split}
\end{equation*}
For $m\ge2$, we have
\begin{equation*}
\begin{split}
f_ig_j-f_jg_i=
\frac{2(i-j)((m-1-j)i+(m-1)(3m+j))}{(i+m)(i+m-1)(i+m+1)(j+m)(j+m+1)(j+m-1)}
>0,\\ 0\le j<i\le m.
\end{split}
\end{equation*}
Further,
\begin{equation*}
\begin{split}
\frac{R_{m,k+1,i}R_{m,k,j}}{(2m+k+1)^2}-\frac{R_{m,k+1,j}R_{m,k,i}}{(2m+k)(2m+k+2)}\\=
\frac{m^{\underline{i}}\,m^{\underline{j}}\,k^{\underline{i-1}}\,k^{\underline{j-1}}\,m^{\underline{k-i}}\,m^{\underline{k-j}}}{(k+1)!k!i!j!}(k+1)
\\\times
\left(\frac{(k+1-j)(m-k+i)}{(2m+k+1)^2}-\frac{(k+1-i)(m-k+j)}{(2m+k)(2m+k+2)}\right)\\=
\frac{m^{\underline{i}}\,m^{\underline{j}}\,k^{\underline{i-1}}\,k^{\underline{j-1}}\,m^{\underline{k-i}}\,m^{\underline{k-j}}}{((k)!)^2\,i!j!}
\\\times
\frac{(2m+k+1)^2(i-j)(m+1)-(k+1-j)(m-k+i)}{(2m+k+1)^2(2m+k)(2m+k+2)}\ge0,
\end{split}
\end{equation*}
since $i-j\ge1$ and $2m+k+1>|m-k+i|$, $2m+k+1>|k+1-j|$, moreover, the inequality is strict if $i,j\in[\max(0,k-m),\min(m,k)]$.
So, the sequence $\frac{A_{k,m}}{B_{k,m}}$, $k=0,\ldots,2m$, increases.
\end{proof}

By Lemmas \ref{Lemma-2} and \ref{Lemma-1}, the function $X_m(x)=F(s)$ increases for $s\in(0,1)$, $x>0$. Therefore, inequality \eqref{L-10} is fulfilled. The proof is completed.

\end{document}